\documentclass[11pt, oneside]{article}   	% use "amsart" instead of "article" for AMSLaTeX format
\usepackage{geometry}                		% See geometry.pdf to learn the layout options. There are lots.
\usepackage{graphicx}				% Use pdf, png, jpg, or eps§ with pdflatex; use eps in DVI mode
\usepackage{amssymb}
\usepackage{amsmath}
\usepackage{amsthm}
\usepackage{harpoon}
\usepackage{accents}

\usepackage{tikz}  
\usepackage{float}
\restylefloat{table}
\usepackage{mathrsfs}
\usepackage{verbatim}
\usepackage{enumitem}  

\usetikzlibrary{positioning,chains,fit,shapes,calc}  

\newtheorem{theorem}{Theorem}
\newtheorem{lemma}{Lemma}
\newtheorem{definition}{Definition}

\newtheorem{remark}{Remark}

\newcommand{\ppp}{\[\begin{aligned}}
\newcommand{\ooo}{\end{aligned}\]}

\begin{document}

\title{On Boolean polynomials and the Union-Closed Conjecture}
\author{Mario DeFranco}
\maketitle

\abstract{For a set of $m$ subsets of a universe set of size $n$, we construct a Boolean polynomial $\mathrm{ICC}_{m,n}(X)$ such that the Union-Closed Conjecture is true for this $m$ and $n$ if and only if $\mathrm{ICC}_{m,n}(X)$ is the zero Boolean polynomial. We use an equivalent formulation, called the Intersection-Closed Conjecture.} 

\section{Introduction} 

Fix integers $m$ and $n$ both greater than 1. Let $[1,n]$ denote the set of integers $\{ j \colon 1 \leq i \leq n\}$, and suppose we have $m$ distinct sets $A_i \subset [1,n], 1 \leq i \leq m$. Let $A$ denote the set of these sets: 
\[
A = \{A_i \colon 1 \leq i \leq m\}.
\]
Also suppose that $A$ is union-closed, i.e.
\[
A_{i_1} \cup A_{i_2}  \in A
\]
for any pair $1 \leq i_1 <i_2 \leq m$. 
Then the Union-Closed Conjecture states that there exists some $j \in [1,n]$ such that 
\[
\#\{ i \colon j \in A_i \} \geq \lceil \frac{m}{2} \rceil.
\]
See \cite{Bruhn} for a history of this conjecture up to 2013. It is also known as Frankl's Conjecture. 

In this paper, we consider the Intersection-Closed Conjecture, a well-known equivalent formulation; instead of $A$ being union-closed, suppose that it is intersection-closed, i.e. 
\[
A_{i_1} \cap A_{i_2}  \in A
\]
for any pair $1 \leq i_1 <i_2 \leq m$. Then the Intersection-Closed Conjecture (statement ICC$(m,n)$) states that there exists some $j \in [1,n]$ such that 
\[
\#\{ i \colon j \in A_i \} \leq \lfloor \frac{m}{2} \rfloor.
\]

In Section 2, we define a Boolean polynomial $\mathrm{ICC}_{m,n}(X)$ such that statement ICC$(m,n)$ is equivalent to this polynomial being the zero Boolean polynomial. Our approach is use the standard encoding of sets as elements in $\mathbb{F}_2^n$ and then to construct characteristic functions for the premises and conclusion of the conjecture. In Section 3, we prove some formulas for certain Boolean polynomials used in our construction. In Section 4, we mention some other problems which can similarly formulated using Boolean polynomials.  

We note that the paper \cite{Lozin} phrases the sets $A_i$ in terms of Boolean functions, but does not construct the characteristic functions of various properties of the sets as done in Section 2 here. 

\section{The Boolean polynomial $\mathrm{ICC}_{m,n}(X)$ }
\begin{definition} 
Fix integers $m,n \geq 2$. Consider a set of $mn$ Boolean variables $$\{x_{i,j} : 1 \leq i \leq m; 1 \leq j \leq n \},$$ i.e. $x_{i,j}^2 = x_{i,j}$. Let $X_i$ denote the sequence  
$$X_i = (x_{i,1}, \ldots, x_{i,n}).$$ 
Let $B$ denote the Boolean algebra $$B=\mathbb{F}_2[\{x_{i,j}\}_{i,j}] .$$

For each specialization $S$ of Boolean variables  to values in $\mathbb{F}_2$, each sequence $X_i$ corresponds to a subset of $\{ 1, \ldots, n\}$, so we will denote such a set by $X_i|_S$. We will also let $u_j, v_j$ for $1 \leq j  \leq n$ denote Boolean variables, and we will denote specializations of $u_j,v_j$ to $\mathbb{F}_2$ by $S$ also. We will call any polynomial of Boolean variables with coefficients in $\mathbb{F}_2$ a \emph{``Boolean polynomial"}. 
\end{definition} 

\begin{definition} For the variables $u=(u_j)_{j=1}^n, v=(v_j)_{j=1}^n$, define the Boolean polynomial $\mathrm{PairwiseDistinct}(u,v)$
$$
\mathrm{PairwiseDistinct}(u,v)=1+\prod_{j=1}^n (1+u_j+v_{j}).
$$

Then the sets $u|_S$ and $v|_S$ are distinct if and only if 
$$\mathrm{PairwiseDistinct}(u,v)|_S=1.$$

\end{definition} 

\begin{definition} 

Define $\mathrm{Distinct}(X)  \in B$ by
$$
\mathrm{Distinct}(X) =\prod_{1 \leq i_1<i_2 \leq m}\mathrm{PairwiseDistinct}(X_{i_1},X_{i_2}). 
$$
\end{definition} 

\begin{definition} 
A set $u|_S$ is a member of the set $X|_S$ if and only if 
$$
\prod_{i=1}^m \mathrm{PairwiseDistinct}(u,X_i)|_S = 0.
$$ 
Define $\mathrm{Member}(X,u)$ by 
$$
\mathrm{Member}(X,u) = 1+ \prod_{i=1}^m \mathrm{PairwiseDistinct}(u,X_i)
$$
\end{definition} 

\begin{definition} 

Define the sequence $u \cap v$ 
$$
u \cap v= (u_jv_j)_{j=1}^n
$$

Thus define the Boolean polynomial $\mathrm{IntersectionClosed}(X)$ by 
$$
\mathrm{IntersectionClosed}(X) = \prod_{1 \leq i_1<i_2 \leq m}\mathrm{Member}(X,X_{i_1} \cap X_{i_2}) 
$$

Therefore the set $X|_S$ is intersection-closed if and only if 
$$
\mathrm{IntersectionClosed}(X)|_S = 1.
$$
\end{definition} 

\begin{definition} 
Define a matching on the set $[1,2k]$ to be a set partition of $[1,2k]$ with each set having size 2. Define a matching on the set $[1,2k+1]$ to be a set partition of $[1,2k+1]$ with one set of size 1 and all other sets of size 2. Let $\mathrm{Matchings}(m)$ denote the set of all matchings of $[1,m]$. For $M \in  \mathrm{Matchings}(m)$, define
$$
H_{M,j}(X) = \begin{cases} \prod_{\{ i_1,i_2\} \in M} (1+ x_{i_1,j}x_{i_2,j}) &\text{ if $m$ is even} \\ 
(1+x_{i_0, j})\prod_{\{ i_1,i_2\} \in M} (1+ x_{i_1,j}x_{i_2,j}) &\text{ if $m$ is odd} 
\end{cases}
$$
where $\{i_0\}$ is the unique set of size 1 in $M$, if $m$ is odd. 

For a specialization $S$ of $X$ and integer $j, 1 \leq j \leq n$, define the statement $H(S,j)$ to be: 
\[
\#\{ i \colon j \in X_i|_S\} \leq  \lfloor \frac{m}{2}\rfloor
\]
 where $m$ is the number of sets in the set $X|_S$. 
\end{definition}

\begin{remark}If statement H$(S,j)$ is true, then there exists some matching $M$ of $[1,m]$ such that  
$$
H_{M,j}(X)|_S = 1.
$$

Thus if statement $H(S,j)$ is true, then 
$$
\prod_{M \in \mathrm{Matchings}(m)} (1+H_{M,j}(X))|_S=0. 
$$
\end{remark}

\begin{definition} 
Define $$\mathrm{HalfMembership}(X) = \prod_{j=1}^n \prod_{M \in \mathrm{Matchings}(m)} (1+H_{M,j}(X)).$$
\end{definition}

\begin{definition} 
Let $\mathrm{Specializations}_{\mathrm{ICC}}(m,n)$ denote the set of specializations $S$ of the variables $x_{i,j}$ such that the set $X|_S$ is an intersection-closed set of distinct sets. Define the statement $\mathrm{ICC}(m,n)$ to be ``For a specialization $S \in \mathrm{Specializations}_{\mathrm{ICC}}(m,n)$, there exists some integer $j$ such that 
\[
\#\{ i \colon j \in X_j|_S \} \leq \lfloor \frac{m}{2}\rfloor.  
\]
\end{definition}

\begin{theorem} 
Statement $\mathrm{ICC}(m,n)$ is true if and only if $\mathrm{ICC}_{m,n}(X)$ is the zoo Boolean polynomial. 
\end{theorem}
\begin{proof}
If the statement ICC$(m,n)$ is true, then for any $S \in \mathrm{Specializations}_{\mathrm{ICC}}(m,n)$. 
we have
$$
\begin{aligned} 
\mathrm{Distinct}(X)|_S&=1\\ 
\mathrm{IntersectionClosed}(X)|_S&=1 \\
\mathrm{HalfMembership}(X)|_S&=0.
\end{aligned}
$$
And if $S \notin \mathrm{Specializations}_{\mathrm{ICC}}(m,n)$, then either 
$$
\mathrm{Distinct}(X)|_S=0 \text{ or } \mathrm{IntersectionClosed}(X)|_S=0. 
$$

So, if statement ICC$(m,n)$ is true, then the polynomial
$$ \label{7}
\mathrm{ICC}_{m,n}(X) = \mathrm{Distinct}(X)
\mathrm{IntersectionClosed}(X)
\mathrm{HalfMembership}(X)
$$
evaluates to 0 for all specializations $S$, and thus must be the zero Boolean polynomial. 

Conversely, suppose $\mathrm{ICC}_{m,n}(X)$ is the zero Boolean polynomial. Then for $S \in \mathrm{Specializations}_{\mathrm{ICC}}(m,n)$, we know 
$$ \mathrm{Distinct}(X)|_S = 
\mathrm{IntersectionClosed}(X)|_S = 1,
$$ so must have 
$$\mathrm{HalfMembership}(X)|_S=0.$$ Thus $H_{M,j}(X)|_S=1$ for some $M$ and $j$, which implies that the element $j$ is in at most $\lfloor \frac{m}{2} \rfloor$ sets of $X|_S$. Thus the statement ICC$(m,n)$ is true if and only if $\mathrm{ICC}_{m,n}(X)$ is the zero Boolean polynomial. This completes the proof.
\end{proof}
\section{Formulas for Member$(X,u)$ and  \newline Member$(X,u)$PairwiseDistinct$(u,v)$}

\begin{definition} 
For a set $W \subset \mathbb{Z}$, let $2^W$ the set of subsets of $W$, and let $(2^W)'$ denote the set of non-empty subsets of $W$. For integer $k \geq 0$, let ${W \choose k}$ denote the set of subsets of $W$ of size $k$, which we write as ordered $k$-tuples 
\[
(i_1, \ldots, i_k).
\]  
Ket $W^c$ denote the complement of $W$ in $[1,n]$. 
%For a sets $W,P \subset [1,n]$, let $W \backslash P$ denote 
%\[
%W \backslash P = \{ i \colon i \in W, i \notin P\}.
%\]
For a set $V$ of Boolean variables, let $e_r(V)$ denote the $r$-th elementary symmetric polynomial on these variables. 
\end{definition}
\begin{lemma} \label{l Member}
\begin{align*}
\mathrm{Member}(X,u) = &\prod_{j=1}^n (1+u_j)\\ 
 &+\sum_{k=1}^{m-1} \sum_{(i_1, \ldots, i_k) \in {[1,n] \choose k}} \sum_{(W_{i_1}, \ldots, W_{i_k}) \in ((2^{[1,n]} )' )^k} F( \cap_{l=1}^k W_{i_l}))\prod_{l=1}^k   \prod_{j \in W_{i_l}} x_{i_l, j}\\ &+\sum_{(W_{1}, \ldots, W_{m}) \in ((2^{[1,n]} )' )^m} (\prod_{j \in (\cap_{l=1}^m W_l)^c} (1+u_j) )
 \prod_{l=1}^m \prod_{j \in W_{l}} x_{l, j}\end{align*}
where for any set $W \subset [1,n]$, 
\begin{align*}
 F(W)= (\prod_{j \in W^c} (1+u_j)  ) \sum_{r=1}^{|W|} e_r(\{ u_j \colon j \in W\})) 
\end{align*}

\end{lemma}

\begin{proof} 
By definition 
\begin{align*}
\mathrm{Member}(X,u) &= 1  + \prod_{i=1}^m (1+ \prod_{j=1}^n (1+u_j +x_{i,j}))\\ 
&= 1 + \prod_{i=1}^m (1+ \sum_{W \subset [1,n]} (\prod_{j \in W} x_{i,j} )\prod_{j \in W^c} (1+u_j) ) \\ 
&= 1+ 1+\sum_{k=1}^m \sum_{(i_1, \ldots, i_k) \in {[1,n] \choose k}} \sum_{(W_{i_1}, \ldots, W_{i_k}) \in ({2^{[1,n]} } )^k} \prod_{l=1}^k (\prod_{j \in W_{i_l}} x_{i_l, j})\prod_{j \in W_{i_l}^c} (1+u_j).\end{align*} 

In the sum over index $k$, at $k$ equal to some $k_0<m$, fix some index $\vec{i} = (i_1, \ldots, i_{k_0})$ and then some index $\vec{W}=(W_{i_1}, \ldots, W_{i_{k_0}})$ such that each of these sets $W_{i_l}$ is non-empty. The contribution from this triple $(k_0, \vec{i}, \vec{W})$ is
\begin{align*} 
&\prod_{l=1}^{k_0} (\prod_{j \in W_{i_l}} x_{i_l, j}\prod_{j \in W_{i_l}^c} (1+u_j)) \\ 
=& ( \prod_{l=1}^{k_0}\prod_{j \in W_{i_l}} x_{i_l, j}) \prod_{j \in \cup_{l=1}^{k_0} W_{i_l}^c} (1+u_j) \\ 
=& (\prod_{l=1}^{k_0}\prod_{j \in W_{i_l}} x_{i_l, j}) \prod_{j \in (\cap_{l=1}^{k_0} W_{i_l})^c} (1+u_j)
\end{align*}
The only other terms that have this same factor of the variables $x_{i,j}$ come from terms with $k=k_1 >k_0$ with indices $(i_1', \ldots, i_{k_1}')$ and $(W_{i_1'}, \ldots, W_{i_{k_1}'})$ such that each $i_l$ is equal to some $i_{l'}'$ with $W_{i_l} = W_{i_{l'}'}$, and the remaining $W_{i_l'}$ are empty. Since $k_1>k_0$, at least one of these sets is empty, so the $u$-contribution is
\[
\prod_{j=1}^n (1+u_j).
\]
Summing over all such $k_1$ and counting the corresponding indices gives the $u$-contribution 
\begin{align*}
&\prod_{j=1}^n (1+u_j)\sum_{k=k_0+1}^m {m-k_0 \choose k-k_0}\\
=& \prod_{j=1}^n (1+u_j) (2^{m-k_0}-1)\\ 
=&  \prod_{j=1}^n (1+u_j) \end{align*}
because $m>k_0$. 
Thus the $u$-coefficient of this $x$-product is 
\[
(\prod_{j \in (\cap_{l=1}^{k_0} W_{i_l})^c} (1+u_j)  ) + \prod_{j=1}^n (1+u_j). 
\]
Factoring the above expression gives the stated form of $ F(\cap_{l=1}^{k_0} W_{i_l})$. 

In the sum over $k$, the terms with $W_{i_l}$ all empty contribute 
\begin{align*}
&\prod_{j=1}^n (1+u_j)\sum_{k=1}^m {m\choose k} \\ 
=& \prod_{j=1}^n (1+u_j)(2^m-1)\\ 
=& \prod_{j=1}^n (1+u_j).
\end{align*}
This completes the proof. 
\end{proof}

\begin{lemma} 
\begin{align*}
&\mathrm{Member}(X,u)\mathrm{PairwiseDistinct}(u,v)  = \sum_{P \in (2^{[1,n]})'} G(P) \\ 
&\prod_{j \in P^c}^n (1+u_j v_j)\\ 
 &+\sum_{k=1}^{m-1} \sum_{(i_1, \ldots, i_k) \in {[1,n] \choose k}} \sum_{(W_{i_1}, \ldots, W_{i_k}) \in ((2^{[1,n]} )' )^k} F( \cap_{l=1}^k W_{i_l}, P)\prod_{l=1}^k   \prod_{j \in W_{i_l}} x_{i_l, j}\\ &+\sum_{(W_{1}, \ldots, W_{m}) \in ((2^{[1,n]} )' )^m} (\prod_{j \in (\cap_{l=1}^m W_l)^c \backslash P} (1+u_jv_j) )  \prod_{l=1}^m \prod_{j \in W_{l}} x_{l, j}\end{align*}
where for any set $W \subset [1,n]$, 
\[
 F(W,P)= (\prod_{j \in W^c \cap P^c} (1+u_j)  ) \sum_{r=1}^{|W \cap P^c|} e_r(\{ u_j v_j \colon j \in W \cap P^c\})) 
\]
and 
\[
G(P) = \prod_{j \in P} (u_j+v_j).
\]
\end{lemma}
\begin{proof} 
By definition 
\begin{align*}
\mathrm{PairwiseDistinct}(u,v) &= 1+\prod_{j=1}^n (1+u_j+v_{j}) \\ 
&= \sum_{P \in (2^{[1,n])'}} G(P).
\end{align*}
Using 
\[
(u_j+v_j)u_j v_j=0,
\]
we obtain for any set $W \subset [1,n]$
\[
G(P)\prod_{j \in W} (1+ u_j v_j) = G(P)\prod_{j \in W \cap P^c} (1+ u_j v_j) 
\]
and 
\[
G(P)e_r(\{ u_j v_j \colon j \in W\}) = G(P)e_r(\{ u_j v_j \colon j \in W \cap P^c\}).
\]
Combining this with Lemma \ref{l Member} gives the formula in the lemma statement. This completes the proof. 
\end{proof}

\section{Other Problems}

We note that Boolean polynomials may be used to formulate other problems such as Twin-Prime Conjecture, the Odd Perfect Number Problem, the Collatz Problem, and Diophantine equations.

\end{document}